\documentclass[10pt,fleqn]{article}

\usepackage{mathptmx}
\usepackage[T1]{fontenc}

\usepackage{latexsym,epsfig,verbatim}
\usepackage{amsmath,amsthm,amssymb,textcomp}
\usepackage{bm}  

\usepackage{url}
\usepackage{color}
\usepackage[colorlinks,citecolor=blue]{hyperref}

\usepackage[symbol*,stable]{footmisc}
\DefineFNsymbols{richard}[math]{
*				{**}
\circ 				{\circ\circ} 
\bullet 			{\bullet\bullet} 	
\vartriangle		{\vartriangle\vartriangle}
\blacktriangle		{\blacktriangle\blacktriangle}
\dagger			{\dagger\dagger}
\ddagger			{\ddagger\ddagger}					
}%
\setfnsymbol{richard}
\renewcommand\footnotemark{}

\usepackage{tikz}
\usetikzlibrary{matrix,arrows,decorations.pathmorphing}

\numberwithin{equation}{section} 

\theoremstyle{plain}

\newtheorem{theorem}{Theorem}
\newtheorem{lemma}[theorem]{Lemma}

\newtheorem{proposition}[theorem]{Proposition}

\newtheorem*{claim}{Claim}

\theoremstyle{definition}

\newcommand{\ball}{B}

\newcommand{\injrad}{\mathrm{injrad}}
\newcommand{\exitexit}{E}
\newcommand{\Manifold}{\mathfrak{M}}

\newcommand{\EP}[2]{\mathcal{E}_{#2}^{#1}}

\newcommand{\gmetric}{g}

\newcommand{\AH}{\mathrm{AH}}
\newcommand{\GF}[1]{\mathrm{AH}(#1)^\circ}
\newcommand{\QF}{\mathrm{QF}}
\newcommand{\qf}{\mathrm{qf}}
\newcommand{\Hom}{\mathrm{Hom}}

\newcommand{\univalent}{\varphi}
\newcommand{\univalentextension}{\Phi}

\newcommand{\PSL}{\mathrm{PSL}}
\newcommand{\domain}[1]{\mathcal{D}^{#1}}
\newcommand{\Mod}{\mathrm{Mod}}

\newcommand{\dee}{\mathrm{d}}

\newcommand{\scalednorm}[1]{\| #1 \|_{}}
\newcommand{\poincare}{\lambda}

\newcommand{\Aconstant}{A}
\newcommand{\AAconstant}{A_{0}}
\newcommand{\Bconstant}{A_1}
\newcommand{\bconstant}{\mathcal{B}_0}
\newcommand{\bbconstant}{\mathcal{B}_1}
\newcommand{\bbbconstant}{\mathcal{B}}
\newcommand{\Cconstant}{18\pi}
\newcommand{\dconstant}{d}
\newcommand{\Dconstant}{A_2}
\newcommand{\Econstant}{A_3}
\newcommand{\mconstant}{m}
\newcommand{\Tconstant}{T}

\newcommand{\constantA}{A_4}
\newcommand{\constantB}{A_5}
\newcommand{\constantC}{A_6}

\newcommand{\TianC}{C_\mathrm{\, Tian}} 
\newcommand{\TianEpsilon}{\epsilon_\mathrm{\, Tian}} 

\newcommand{\co}{\colon\thinspace}

\newcommand{\HH}{\mathbb{H}}
\newcommand{\RR}{\mathbb{R}}

\newcommand{\CC}{\mathbb{C}}

\newcommand{\calQ}{\mathcal{Q}}
\newcommand{\Teich}{\mathcal{T}}
\newcommand{\UU}{\mathcal{U}}

\newcommand{\dd}{\mathrm{d}}

\newcommand{\BigO}{\mathcal{O}}
\newcommand{\floor}[1]{\lfloor #1 \rfloor}

\newcommand{\ee}{\mathbf{e}}

\newcommand{\uu}{\mathbf{u}}

\newcommand{\vv}{\mathbf{v}}
\newcommand{\ww}{\mathbf{w}}
\newcommand{\mm}{\mathbf{m}}
\newcommand{\nn}{\mathbf{n}}

\newcommand{\ii}{\mathbf{i}}
\newcommand{\jj}{\mathbf{j}}

\newcommand{\ddxk}{\frac{\partial}{\partial x^k}}
\newcommand{\ddxl}{\frac{\partial}{\partial x^\ell}}

\newcommand{\ddt}{\frac{\partial}{\partial t}}
\newcommand{\ddx}{\frac{\partial}{\partial x}}
\newcommand{\ddy}{\frac{\partial}{\partial y}}

\newcommand{\dduone}{\frac{\partial}{\partial u^1}}
\newcommand{\ddutwo}{\frac{\partial}{\partial u^2}}
\newcommand{\dduthree}{\frac{\partial}{\partial u^3}}

\newcommand{\grad}{\nabla}

\newcommand{\Der}{\mathrm{D}}

\newcommand{\Schwarz}{\mathcal{S}\!}

\newcommand{\normalcurv}{\mathcal{N}}

\newcommand{\Ric}{R}
\newcommand{\Ricci}{\mathrm{Ric}}

\newcommand{\infinity}{{\rotatebox{90}{\hspace{-.75pt}\large $8$}}}
\newcommand{\smallinfinity}{{\hspace{.5pt} \rotatebox{90}{\footnotesize $8$}}}
\newcommand{\scriptinfinity}{{\rotatebox{90}{\scriptsize $8$}}}

\begin{document}

\title{
\textbf{Thick--skinned $3$--manifolds}
}
\author{Richard P. Kent IV and Yair N. Minsky\thanks{This research was partially supported by NSF grants DMS--1104871 and DMS--1005973.
The authors acknowledge support from U.S. National Science Foundation grants DMS--1107452, 1107263, 1107367 ``RNMS: GEometric structures And Representation varieties'' (the GEAR Network).}
}

\date{June 30, 2014}

\maketitle

\begin{abstract}
We show that if the totally geodesic boundary of a compact hyperbolic $3$--manifold $M$ has a collar of depth $d \gg 0$, then the diameter of the skinning map of $M$ is no more than $Ae^{-d}$ for some $A$ depending only on the genus and injectivity radius of $\partial M$.
\end{abstract}

\bigskip


\noindent
Given a discrete group $G$, we equip $\Hom(G, \PSL_2(\CC))$ with the compact--open topology.  
This induces a topology on the space $\Hom(G, \PSL_2(\CC))/\PSL_2(\CC)$ of conjugacy classes of representations called the \textit{algebraic topology}.
If $N$ is a connected $3$--manifold, we let 
\[
\AH(N) \subset
\Hom(\pi_1(N), \PSL_2(\CC))/\PSL_2(\CC)
\]
be the subset of conjugacy classes of discrete and faithful representations with the subspace topology.
Each such conjugacy class corresponds to a hyperbolic structure on a $3$--manifold homotopy equivalent to $N$.
Let $\GF{N}$ be the interior of $\AH(N)$.

Let $S$ be a closed connected oriented surface of negative Euler characteristic.\footnote{We assume that $S$ is connected for convenience.  Defining the Teichm\"uller space of a disconnected surface to be the product of the Teichm\"uller spaces of the components, and adapting the proofs accordingly, we could dispense with this assumption.}
By work of Marden \cite{marden} and Sullivan \cite{sullivanstability}, the space $\GF{S \times \RR}$ equals the set $\QF(S)$ of convex cocompact, or \textit{quasifuchsian}, hyperbolic structures on $S \times \RR$.
By the Simultaneous Uniformization Theorem \cite{Bers.1960}, the space $\QF(S)$  is naturally homeomorphic to the product of Teichm\"uller spaces $\Teich(S) \times \Teich(\overline{S})$.
If $(X,\overline{Y})$ is a point of $\Teich(S) \times \Teich(\overline{S})$, we let $\qf(X, \overline{Y})$ denote $S \times \RR$ with the corresponding convex cocompact hyperbolic structure.

Let $M$ be a compact hyperbolic $3$--manifold with totally geodesic boundary homeomorphic to $S$.
A generalization of the Simultaneous Uniformization Theorem due to Ahlfors, Bers, Marden, and Sullivan (see \cite{Bers.1970} or \cite{Canary.McCullough.2004}) tells us that the space $\GF{M}$ of convex cocompact hyperbolic metrics on $M^\circ$ is naturally homeomorphic to the Teichm\"uller space $\Teich(S)$.
If $X$ is a point in $\Teich(S)$, we let $M^X$ denote $M^\circ$ equipped with the corresponding convex cocompact hyperbolic structure.

The inclusion  $\partial M \to M$ induces a map $\GF{M} \to \QF(S)$.
Identifying $\GF{M}$ with $\Teich(S)$ and $\QF(S)$ with $\Teich(S) \times \Teich(\overline{S})$, this map is given by $X \mapsto (X, \sigma_M(X))$ for some function
\[
	\sigma_M \co \Teich(S) \to \Teich(\overline{S}).
\]
The function $\sigma_M$ is Thurston's \textit{skinning map} associated to $M$.
This map is a key ingredient in Thurston's proof of Geometrization for Haken Manifolds \cite{Morgan.1979,Otal.1998,Otal.2001,Kapovich.2001}.
Thurston's Bounded Image Theorem \cite{Thurston.1979.Bangor,Kent.2010} states that the image of $\sigma_M$ is bounded, and we call the diameter of the image  with respect to the Teichm\"uller metric the \textit{diameter of $\sigma_M$}. 
In \cite{Kent.2010}, the first author proved that if $\partial M$ has a large collar, then $\sigma_M$ carries a large ball to a set of very small diameter (Theorem 29 there).
We greatly improve that theorem here.

We say that the totally geodesic boundary $\partial M$ in a hyperbolic $3$--manifold 
$M$ has a \textit{collar of depth $d$} if the $d$--neighborhood of $\partial M$ is homeomorphic to $\partial M \times [0,d]$.
\begin{theorem}\label{DiameterBound.theorem}
If $\epsilon$ and $\mconstant$ are positive numbers, then there are positive numbers $\Aconstant$ and $\Tconstant$ such that the following holds. 
If $M$ is a compact hyperbolic $3$--manifold with totally geodesic boundary $Y$ with $\chi(Y) \geq -\mconstant$ and $\injrad(Y) \geq \epsilon$, and $M$ contains a collar of depth $\dconstant \geq \Tconstant$ about $\Sigma$, then the skinning map $\sigma_M$ has diameter less than $\Aconstant e^{- \dconstant}$.
\end{theorem}

We pause to sketch the proof.

Consider the hyperbolic manifolds $\qf(X,\overline{Y})$ and $\qf(Y,\overline{Y})$.
Very far out toward their $\overline{Y}$--ends, these manifolds are very nearly isometric. 
In fact, the proximity of the metrics decays exponentially in the distance from the convex core.  
Using foliations constructed by C. Epstein, the metrics near the $\overline{Y}$--ends of $\qf(X,\overline{Y})$ and $\qf(Y,\overline{Y})$ may be written down explicitly in terms of the Schwarzian derivatives associated to the projective structures on $\overline{Y}$, see Section \ref{Epstein.surfaces.section}.
This allows us to explicitly glue the ``$X$--side'' of $\qf(X,\overline{Y})$ to the ``$\overline{Y}$--side'' of $\qf(Y,\overline{Y})$ to obtain a metric $\eta$ on $S \times \RR$ which is hyperbolic away from a shallow gluing region of the form $S \times [n,n+1]$, see Figure \ref{GlueQF.figure}.
Calculations (in Sections \ref{Curvature.section.1} and \ref{Curvature.section.3}) show that the resulting metric has  sectional and traceless Ricci curvatures exponentially close to $-1$ and $0$, respectively.
Moreover, the $L^2$--norm of the traceless Ricci curvature of this metric is exponentially small (see Section \ref{Curvature.section.2}).

\begin{figure}
\bigskip
\medskip
\quad \quad \quad \ \
\resizebox{.9\textwidth}{!}{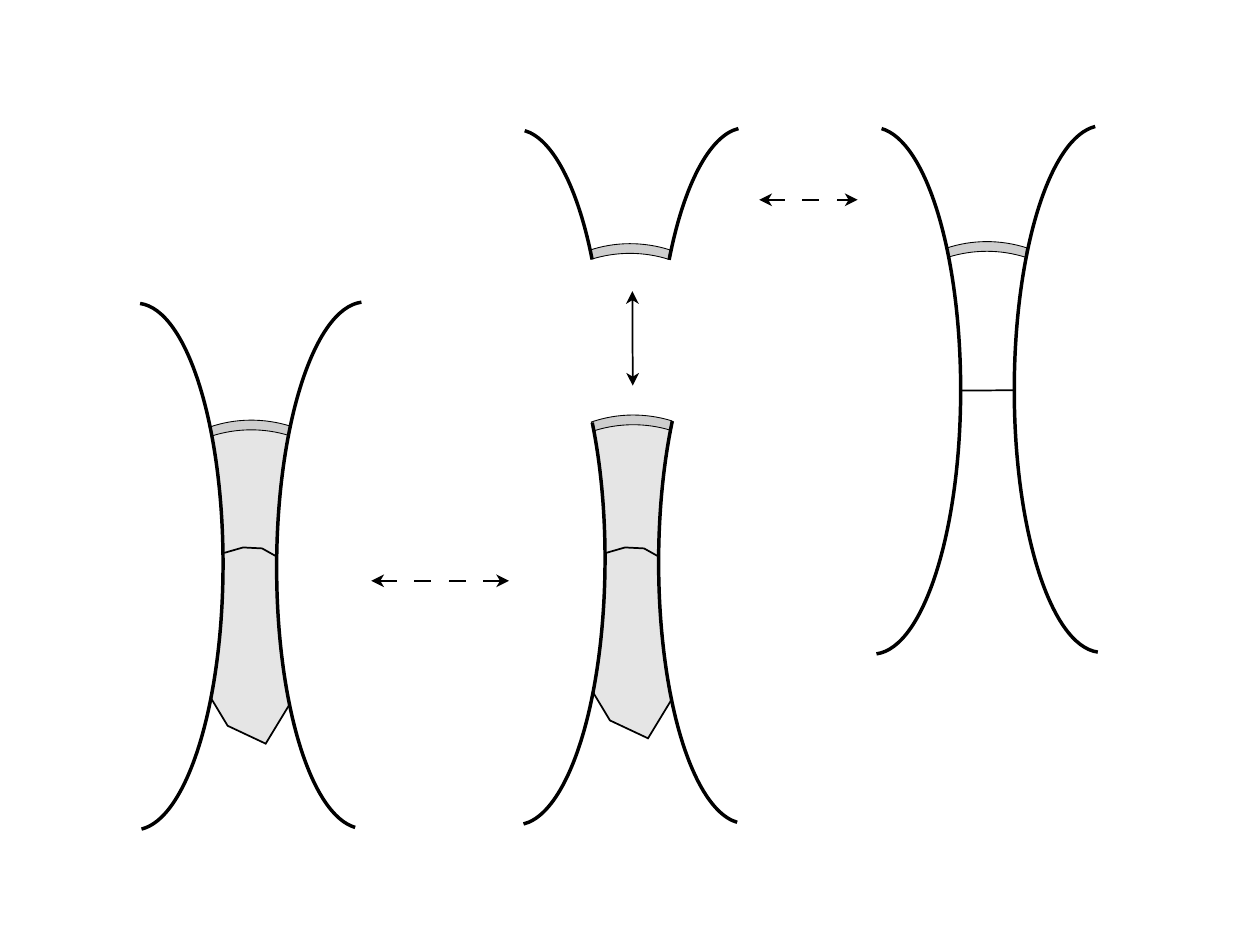}
\caption{Building the metric $\eta$ on $S \times \RR$.}\label{GlueQF.figure}
\end{figure}

Given a hyperbolic manifold $M^Y$ with totally geodesic boundary $Y$ possessing a large collar about its boundary, we may glue the ``$X$--side'' of $\qf(X,\overline{Y})$ to a compact piece of $M^Y$ in the same way to obtain a metric $\omega$ on $M^\circ$ with the same curvature bounds, see Figure \ref{GlueQFMY.figure}.

It is a theorem of Tian that a Riemannian metric on a closed $3$--manifold whose sectional curvatures are very close to $-1$ and whose traceless Ricci curvature has very small $L^2$--norm is $\mathcal{C}^2$--close to a hyperbolic metric, where the proximity depends only on the curvatures and their norms and not the topology of the $3$--manifold.
As our manifold is noncompact, Tian's theorem is not directly applicable.
A theorem of Brooks \cite{Brooks.1986} allows us to circumvent this problem by performing a small quasiconformal deformation of $\qf(X, \overline{Y})$ to obtain a manifold covering a closed one, and we find that $\omega$ is $\mathcal{C}^2$--close to the convex cocompact manifold $M^X \cong M^\circ$ with conformal boundary $X$.
See Section \ref{Proof.section}.

\begin{figure}
\bigskip
\bigskip
\bigskip
\medskip
\quad \quad \quad
\resizebox{.85\textwidth}{!}{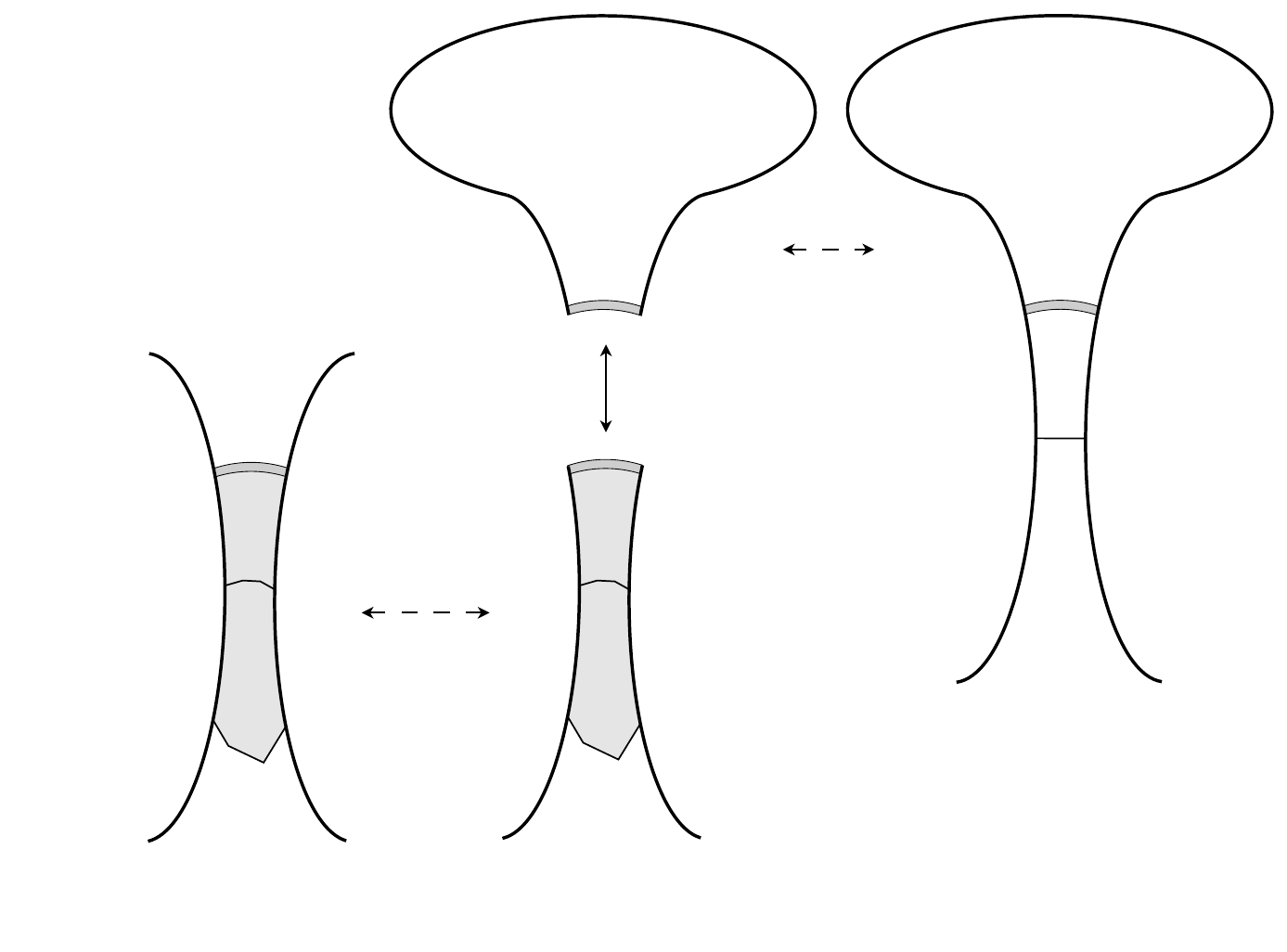}
\caption{Building the metric $\eta$ on $M^\circ$.}\label{GlueQFMY.figure}
\end{figure}

Now, the copy of $S \times \{n+1\}$ in $(M^\circ, \omega)$ is conformally equivalent to $\overline{Y}$, and, for large $n$, the proximity of the metrics implies that the corresponding surface in $M^X$ is very close to both $\sigma_M(X)$ and $\overline{Y}$ in Teichm\"uller space.
We conclude that the image of $\sigma_M$ lies in a small neighborhood of $\overline{Y}$.

\medskip
\noindent \textbf{Generalizations.}
The hypothesis bounding the injectivity radius is needed in the proof in two places: to bound the Sobolev norms of the Schwarzian derivatives of developing maps of the quasifuchsian manifolds, which is needed in the estimation of the curvatures of our model metric; and to guarantee that our model metric is Einstein on the $1$--thin part, which is required by Tian's theorem.  
We suspect that one may dispense with this hypothesis, though our proof can not.

The dependence of the constants $\Aconstant$ and $\Tconstant$ on $\chi(\partial M)$ is likely necessary, though we do not have a sequence of examples to demonstrate this. 
In \cite{Kent.2010}, the first author shows that there are manifolds with arbitrarily large diameter skinning maps, but the maximum depth of the collars about their geodesic boundaries tends to zero.

Skinning maps are defined for any orientable hyperbolic manifold with incompressible boundary, and we in fact establish the following generalization of Theorem \ref{DiameterBound.theorem}.

If $Z$ is a Riemann surface, let $\injrad(Z)$ denote the injectivity radius of the hyperbolic metric in its conformal class.
\begin{theorem}\label{DiameterBound.Generalization.theorem}
Let $M = M^W$ be a convex cocompact hyperbolic $3$--manifold with freely indecomposable nonelementary fundamental group and conformal boundary $W$.
Let
$
p \co \qf(W, \overline{Y}) \to M
$
be the covering map corresponding to $W$.
There are constants $A$, $T$, and $K$ depending only on $\chi(W)$ and $\injrad(\overline{Y})$ such that the following holds.

If $d > T$ and $p$ \textbf{embeds} the $d$--neighborhood of the convex core of $\qf(W,\overline{Y})$  isometrically into $M$, then $\sigma_M$ has diameter no more than $A
e^{-d}$ and the $Kd$--neighborhood of the convex core of $\qf(V, \sigma_M(V))$ isometrically embeds into $M^V$ for any $V$. 

In particular, the manifold $M$ is acylindrical and the totally geodesic boundary has a large collar (of depth $Kd$) in the corresponding hyperbolic structure on $M$.
\end{theorem}
\noindent Note that $M$ being acylindrical would follow from \textit{any} bound on the diameter of the skinning map, as cylindrical manifolds always have unbounded skinning maps.

The proof of Theorem \ref{DiameterBound.Generalization.theorem} roughly follows the sketch given above.
The construction of the Riemannian metric $\eta$ in the general setting is recorded below in Theorem \ref{Model.manifold.theorem}, and may be of independent interest.

\medskip
\noindent \textbf{Miscellaneous notation.}
If $f$ and $F$ are functions of $t$, we use the Landau notation $f = \mathcal{O}(F)$ to mean that there is a constant $L$ such that $|f(t)| \leq LF(t)$ for all $t$.
If $a$, $b$, $c$, $\ldots$ are objects, we write $f(t) = \mathcal{O}_{a,b,c,\ldots}(F(t))$ if  $|f(t) | \leq L F(t)$ for a constant $L$ depending only on $a, b, c, \ldots$
We use the standard notation $W^{k,p}(\mathcal{X})$ for Sobolev spaces and follow the Einstein summation convention.

\medskip
\noindent \textbf{Acknowledgments.}  The authors thank
Ken Bromberg,
David Dumas,
Hossein Namazi,
Sean Paul,
and
Jeff Viaclovsky
for helpful conversations. 
They also thank the referee for such careful readings and numerous useful suggestions.
In particular, we are very grateful for the referee's many careful observations that  improved our estimate from $Ae^{-d/2}$ to $Ae^{-d}$.

\section{Tian's theorem}

\begin{theorem}[Tian \cite{Tian.1990}]
There are numbers $\TianC \geq 1$  and $\TianEpsilon > 0$ such that the following holds.
If $\epsilon < \TianEpsilon$ and $(M,\omega)$ is a closed Riemannian $3$--manifold with sectional curvatures pinched between $-1 - \epsilon$ and $-1 + \epsilon$,  traceless Ricci curvature $\Ricci^\omega + 2\omega = 0$ on the $1$--thin part, and 
\begin{equation}
	\sqrt{
		\int_M \, \| \Ricci^\omega + 2\omega \|_\omega^2 \ \dee V_\omega 
	\ }
	\leq \epsilon,
\end{equation}
then $M$ has a hyperbolic metric $\zeta$ such that 
$
	\smash{\| \omega - \zeta  \|_{\mathcal{C}^{2}(M,\omega)} \leq \TianC \epsilon},  
$
where\linebreak ${\| \, \cdot \,  \|_{\mathcal{C}^{2}(M,\omega)}}$ is the $\mathcal{C}^{2}$--norm with respect to $\omega$.
\qed
\end{theorem}

\noindent
The background metric $\omega$ defines the pointwise norm of any tensor on $M$, and the pointwise $\mathcal{C}^2$--norm of a smoothly varying bilinear form $b$ is defined by taking the supremum of the norm of $b$ and its first two covariant derivatives with respect to $\omega$.
The norm $\| \, b \, \|_{\mathcal{C}^{2}(M,\omega)}$ is then the supremum of the pointwise norms.

\section{Hyperbolic metrics and Epstein surfaces}\label{Epstein.surfaces.section}

Most of this section is a review of Section 6.1 of \cite{Bromberg.2004.JAMS} and Sections 3.2--3.4 of \cite{Anderson.1998}. 

Let $\Delta$ be the open unit disk in $\CC$ parameterized by the variable $z = x + iy$.  
We model hyperbolic space $\HH^3$ as $\Delta \times \RR$ with the metric $\gmetric$ given by 
\begin{equation}\label{Hyperbolic.Metric.equation}
	\dee s^2 
		= \frac{4\cosh^2 t}{(1-|z|^2)^2}  \, \dee x^2 
			+\frac{4\cosh^2 t}{(1-|z|^2)^2} \,  \dee y^2 
			+ \ \dee t^2.
\end{equation}
Note that $\Delta \times \{0\}$ is a totally geodesic hyperbolic plane.
We encode $g$ in the matrix
\[
	g = (g_{ij}) =
		\left(
\begin{tikzpicture}[>=to, line width = .075em, baseline=(current bounding box.center)]
\matrix (m) [matrix of math nodes,  text height=1.5ex, text depth=0.25ex]
{
\poincare^2 \cosh^2 t & 0 & 0 \\
0 & \poincare^2 \cosh^2 t & 0 \\
0 & 0 & 1 \\
};
\end{tikzpicture}
		\right)
\]
where $\poincare = 2/(1-|z|^2)$.

Let $\Gamma^Y$ be a fuchsian group uniformizing $\overline{Y}$ via $\Delta/\Gamma^Y = \overline{Y}$.
This gives us the fuchsian hyperbolic $3$--manifold $\qf(Y,\overline{Y}) = \HH^3/\Gamma^Y$, and a we have a local expression for the hyperbolic metric on $\qf(Y,\overline{Y})$ in \eqref{Hyperbolic.Metric.equation}. 
We let $\domain{Y}$ be a compact fundamental domain for $\Gamma^Y$ acting on $\Delta$ whose interior contains zero.

\newcommand{\BB}{\mathbb{B}}

We also want to consider the Poincar\'e ball model $\BB^3$ of $\HH^3$, with boundary
the Riemann sphere $\widehat\CC$. There is a unique isometry
$\iota:\Delta\times\RR \to \BB^3$ which extends continuously to
$\Delta\times\{\pm\infinity\}$, taking $\Delta\times\{\infinity\}$ to
  $\Delta\subset\widehat\CC$ by the identity map.
  Note that on $\Delta\times\{\infinity\}$ this extension carries $\ddt$  to a vector pointing \textit{downward}, or \textit{out} of hyperbolic space. 

Let $\univalent \co \Delta \to \widehat \CC$ be a univalent function,\footnote{A function $\Delta \to \widehat \CC$ is \textit{univalent} if it is injective and holomorphic.} let
\[
\Schwarz \univalent(z) 	=
					\left( \frac{\univalent_{zz}}{\univalent_z}\right)_{\!z} - \frac{1}{2} \left(\frac{\univalent_{zz}}{\univalent_z}\right)^{\!2}
\]
be its Schwarzian derivative, and let
$\scalednorm{\Schwarz \univalent(z)} = | \, \poincare^{-2} \Schwarz
\univalent(z)\, |$. 
Let $M_{\univalent(z)}:\widehat\CC\to\widehat\CC$ be the osculating M\"obius transformation to $\univalent$ at $z$ (the M\"obius transformation with the same $2$--jet as $\univalent$ at $z$).
This uniquely extends to an isometry $M_{\univalent(z)} \co \BB^3 \to \BB^3$.
There is then a map $\Phi \co \Delta\times\RR \to \BB^3$ given by 
$$\Phi(z,t) = M_{\univalent(z)}(\iota(z,t)),$$
which also admits a continuous extension to $\Delta\times(-\infinity,\infinity]$ with $\Phi(z,\infinity) = \univalent(z)$.
We henceforth identify $\Delta$ with $\Delta \times \{\infinity\}$,
and identify both $\Delta\times\RR$ and $\BB^3$ with $\HH^3$.

There is an orthonormal basis $\ee_1$, $\ee_2$, $\ddt$ for the tangent
space to $\HH^3$ at $(z,t)$ and an orthonormal basis for the tangent
space to $\HH^3$ at $\Phi(z,t)$ such that the derivative of $\Phi$ at
$(z,t)$ is given by
\begin{equation}\label{Derivative.Phi.equation}
	\Der \Phi \big|_{(z,t)} =
	\left(
\begin{tikzpicture}[>=to, line width = .075em, baseline=(current bounding box.center)]
\matrix (m) [matrix of math nodes,  text height=1.5ex, text depth=0.25ex]
{
1 + \frac{\scalednorm{\Schwarz \univalent(z)}}{e^t \cosh t} & 0 & 0 \\
0 & 1 - \frac{\scalednorm{\Schwarz \univalent(z)}}{e^t \cosh t} & 0 \\
0 & 0 & 1 \\
};
\end{tikzpicture}
	\right).
\end{equation}
(In \cite{Bromberg.2004.JAMS}, the eigenvalues of the matrix $\Der \Phi - I$ are off by a factor of $4$.)
If we normalize (by conjugation in $\PSL_2\CC$) so that $z=0$ and so that the osculating M\"obius transformation at zero is the identity, we have 
\begin{align}
	\label{Basis.equation}
	\begin{split}
	2 \cosh t \cdot
	\ee_1 	&= \phantom{-}\cos (\theta_0) \ddx + \sin  (\theta_0)  \ddy \\
	2 \cosh t \cdot
	\ee_2 	&= -\sin (\theta_0) \ddx + \cos (\theta_0) \ddy 
	\end{split}
\end{align}
where $\theta_0$ is the argument of $\Schwarz \univalent(0)$, see Section 3.3 of \cite{Anderson.1998}.

The inequality $\| \Schwarz \univalent(z) \| \leq 3/2$ holds for univalent $\univalent$ by a celebrated theorem of Kraus \cite{Kraus.1932} and Nehari \cite{Nehari.1949}, and so $\Phi$ is an orientation--preserving immersion on $\{(z,t) \in \Delta \times \RR \ | \ t  > \log\sqrt{2} \}$, by \eqref{Derivative.Phi.equation}.

The principal curvatures of $\univalentextension(\Delta \times \{t\})$ at $(z,t)$ are given by
\begin{equation}\label{PrincipalCurvatures}
	\kappa_\pm(z,t) = \frac{1 - (1 \pm 2 \| \Schwarz \univalent(z) \|)e^{-2t}}{1 + (1 \pm 2 \| \Schwarz \univalent(z) \|)e^{-2t}} ,
\end{equation}
when this is defined, see Proposition 6.3 of \cite{Bromberg.2004.JAMS}.
If $\Schwarz \univalent(z) = 1$, then $\kappa_+(z,t) = \coth t$ by continuity of the principal curvatures.
These curvatures are positive provided $t > \log 2$, thanks to the Kraus--Nehari theorem, and so $\univalentextension(\Delta \times \{t\})$ is locally convex for such $t$.\footnote{The statement in \cite{Bromberg.2004.JAMS} that the $\univalentextension(\Delta \times \{t\})$ are convex when $t > 0$ is an error.}

We now specialize to univalent $\univalent$ associated to ends of hyperbolic $3$--manifolds.

Let $M$ be a complete hyperbolic $3$--manifold with conformally compact incompressible end $\exitexit = \exitexit^M$ compactified by the Riemann surface $\overline{Y}$. 
Let $\Gamma = \Gamma^M \subset \PSL_2(\CC)$ be a Kleinian group uniformizing $M$.

Blurring the distinction between ends and their neighborhoods, the end $\exitexit$ is homeomorphic to $S \times (0,\infinity)$, and we pick isomorphisms $\pi_1(E) \leftarrow \pi_1(S) \to \Gamma^{\overline{Y}}$  compatible with the chosen marking of $\overline{Y}$.
Choose a component $\UU^{\overline{Y}}$ of the domain of discontinuity of $\Gamma^M$ corresponding to $\exitexit$, let $\univalent \co \Delta \to \UU^{\overline{Y}}$ be the $\pi_1(S)$--equivariant Riemann mapping, and let $\Phi \co \Delta \times (\log 2, \infinity] \to \HH^3 \cup \UU^{\overline{Y}}$ be as above.
By the above discussion, the map $\Phi$ is an immersion and the surfaces $\Phi(\Delta \times \{t\})$ are locally strictly convex.

Given a smooth surface $\mathcal{F}$ and a locally strictly convex immersion $f \co \mathcal{F} \to \HH^3$, there is an associated \textit{Gauss map} $\mathfrak{g} \co\mathcal{F} \to \widehat{\CC}$, defined as follows.
For any $w$ in $\mathcal{F}$, there is a neighborhood $U$ of $w$ on which $f$ is an embedding.
Since $f$ is locally strictly convex, there is a unique geodesic ray emanating from $f(w)$ that is perpendicular to $f(U)$ and moves away from the center of curvature.
This geodesic ray has a unique endpoint $g(w)$ in $\widehat{\CC}$, and this defines a  map $\mathfrak{g} \co \mathcal{F} \to \widehat{\CC}$.

Note that it follows from (\ref{Derivative.Phi.equation}) that the
Gauss map $\mathfrak{g}_t \co \Delta \times \{t\} \to \widehat{\CC}$
associated to the immersion $\Phi_t = \Phi|_{\Delta \times \{t\}}$ is
equal to the embedding $\univalent \circ p_t$, where $p_t \co \Delta \times
\{t\} \to \Delta$ is the projection $p_t(z,t) = z$.  
In particular,
$\mathfrak{g}_t$ descends to an injective map on $\Delta/\pi_1(S)$.

Let
$
\Xi \co \Delta/\pi_1(S) \times (\log 2, \infinity] \to \exitexit \cup \overline{Y}
$
be the immersion induced by $\Phi$ and let $\Xi_{\,t}$ be the restriction of $\Xi$ to $\Delta/\pi_1(S) \times \{t\}$.

We claim that $\Xi$ is a diffeomorphism.\footnote{The argument given here is implicit in Sections 3 and 6 of \cite{Bromberg.2004.JAMS}.}

To see this, first note that since $\Xi_{\,\scriptinfinity}$ is a diffeomorphism and $S$ is compact, the implicit function theorem provides a $t_0$ such that $\Xi$ is an embedding when restricted to $\Delta/\pi_1(S) \times [t_0,\infinity]$.

Suppose that $\Xi$ is not a diffeomorphism, and let $t_1$ be the largest $t$ in $(\log 2, \infinity)$ such that $\Xi$ is not injective on $\Delta/\pi_1(S) \times [t,\infinity)$. 
So $\Xi_{\,t_1}$ is not injective, and we have points $a$ and $b$ in $\Delta/\pi_1(S) \times \{t_1\}$ for which  $\Xi_{\,t_1}(a) = \Xi_{\,t_1}(b)$.
Local strict convexity implies that $\Xi_{\,t_1}$ must have a self--tangency at these points, or else there would be a slightly later time $t_2$ at which $\Xi_{t_2}$ failed to be an embedding.
Furthermore, the normal vectors pointing away from the centers of curvature must agree, or else there would again be a slightly later time when $\Xi_t$ failed to embed.
Lifting the map $\Xi_{\,t_1}$ to the map $\Phi_{\,t_1}$, we find distinct points $\widetilde{a}$ and $\widetilde{b}$ in $\Delta \times \{t_1\}$ such that $\mathfrak{g}_{t_1}(\widetilde{a}) = \mathfrak{g}_{t_1}(\widetilde{b})$, contradicting injectivity of the Gauss map $\mathfrak{g}_{t_1}$.
We conclude that $\Xi$ is a diffeomorphism.

Let $\EP{}{t} = \EP{M}{t}$ be the image of $\Xi_{\,t}$.
We call the $\EP{}{t}$ \textit{Epstein surfaces}, in honor of their study by C. Epstein \cite{Epstein.1984}, who calls them \textit{Weingarten surfaces}.

\section{Gluing hyperbolic metrics}

We prove the following gluing theorem for hyperbolic manifolds, which says roughly the following.
If two hyperbolic manifolds $M^1$ and $M^2$ contain separating product regions $S \times I$ isometric to products taken from far out in conformally compact $\overline{Y}$--ends of two other hyperbolic manifolds $N^1$ and $N^2$, then these regions cut $M^1$ and $M^2$ into convex pieces $\mathcal{A}^1$ and $\mathcal{A}^2$ and concave pieces $\mathcal{B}^1$ and $\mathcal{B}^2$.
The theorem says that one may glue $\mathcal{A}^1$ to $\mathcal{B}^2$ so that the resulting manifold admits a Riemannian metric that is very nicely behaved near the gluing site and hyperbolic elsewhere.  
In particular, if the resulting manifold is closed, this Riemannian metric satisfies the hypotheses of Tian's theorem, and is thus close to the unique hyperbolic metric.
See Figure \ref{GlueGeneral.figure}.

\begin{figure}
\bigskip
\bigskip
\bigskip
\bigskip
\bigskip
\bigskip
\resizebox{.59\textwidth}{!}{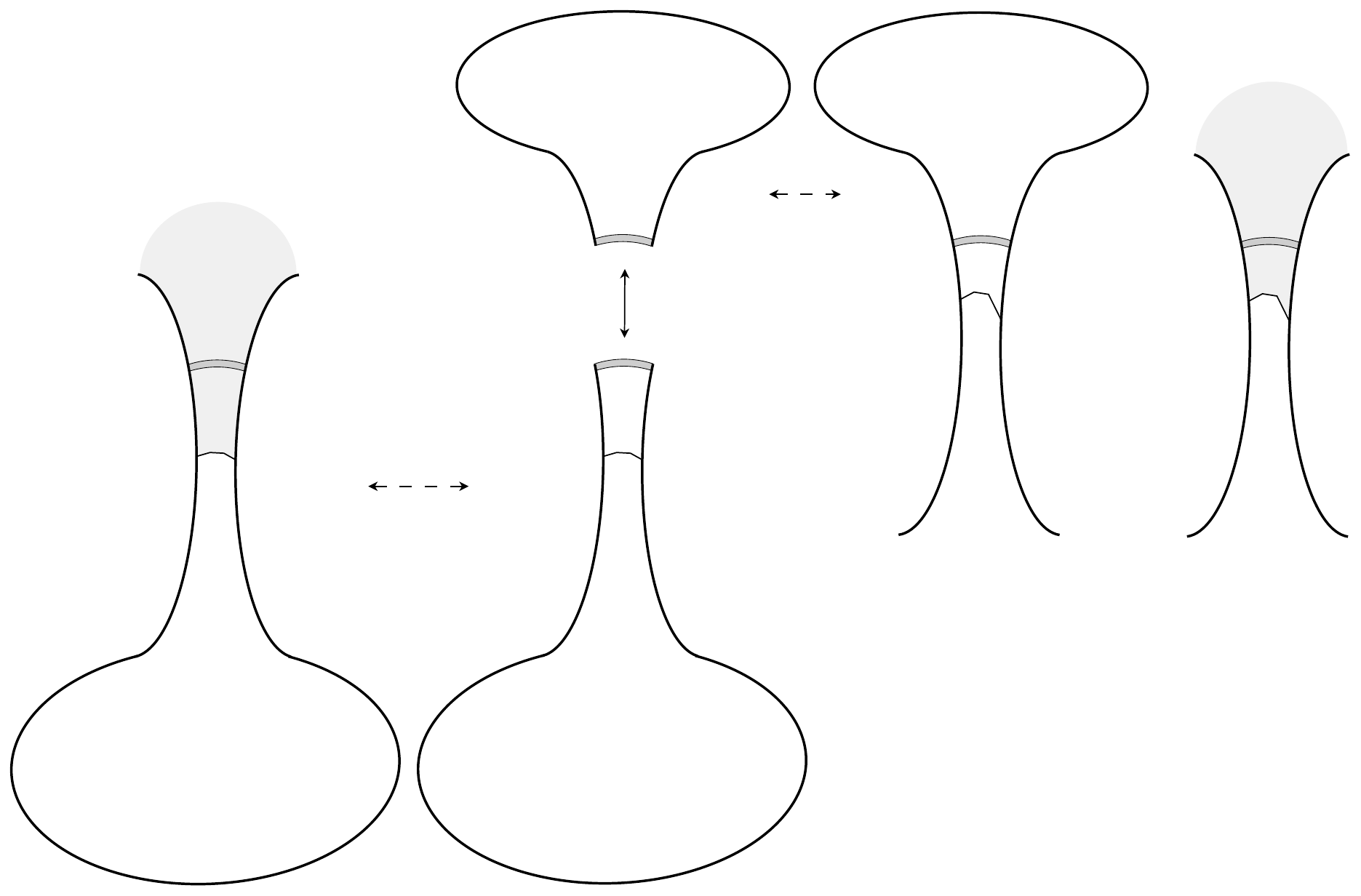}
\caption{Building $\Manifold$ with the metric $\eta$. The case $M^1=N^1$ is shown for simplicity.}\label{GlueGeneral.figure}
\end{figure}

\begin{theorem}[Model manifold for convex cocompact gluings]\label{Model.manifold.theorem}
Given positive numbers $\epsilon$ and $m$, there is a constant $\AAconstant = \AAconstant(m, \epsilon)$ such that the following holds.

Let $\overline{Y}$ be a marked Riemann surface with $\chi(\overline{Y}) \geq - m$ and $\injrad(\overline{Y}) \geq \epsilon$.

For $i$ in $\{1,2\}$, let $(N^i, \zeta^i)$ be a complete hyperbolic $3$--manifold with a conformally compact (marked) incompressible end $\exitexit^i$ conformally compactified by $\overline{Y}$.
Let $\EP{i}{t}$ be the foliation of $\exitexit^i$ by Epstein surfaces.
For $\log 2 < a < b$, let $\exitexit^i_{[a,b]}$ be the compact region of $\exitexit^i$ bounded by $\EP{i}{a}$ and $\EP{i}{b}$.

Let $(M^i, \xi^i)$ be a complete hyperbolic manifold containing a separating submanifold isometric to $\exitexit^i_{[n,n+3]}$.
For $t$ in $[a,b]$, let $\mathcal{A}^i_t$ be the closure of the convex component of $M^i - \EP{i}{t}$, and let $\mathcal{B}^i_t$ be the closure of the concave component.

Let $\Manifold$ be the topological manifold obtained by gluing $\mathcal{A}^{1}_n$ to $\mathcal{B}^2_n$ along their boundaries respecting the marking of $\overline{Y}$.

Then there is a Riemannian metric $\eta$ on $\Manifold$ that satisfies the following.
\begin{enumerate}
\item\label{Model.manifold.theorem.item.einstein} The inclusions of $(\mathcal{A}^1_n, \xi^1)$ and $(\mathcal{B}^2_{n+3}, \xi^2)$ into $(\Manifold, \eta)$ are isometric embeddings.
\item\label{Model.manifold.theorem.item.sectional} The sectional curvatures of $\eta$ are within $\AAconstant e^{-2n}$ of $-1$.
\item\label{Model.manifold.theorem.item.twonorm} The $L^2$--norm of the traceless Ricci curvature of $\eta$ is no more than $\AAconstant e^{-n}$. 
\item\label{Model.manifold.theorem.item.radius} The injectivity radius of $\eta$ on $\exitexit^2_{[n,n+3]}$ is at least $(1- \AAconstant e^{-2n})\injrad(\overline{Y})\cosh n$.
\end{enumerate}
\end{theorem}

\begin{proof}
We begin by assuming that $M^1 = N^1$ and that $M^2 = N^2 = \qf(Y,\overline{Y})$.

\subsection{The proof of Theorem \ref{Model.manifold.theorem} when $M^1 = N^1$ and that $M^2 = N^2 = \qf(Y,\overline{Y})$.}

Let $s_0(t)$ be a smooth nonincreasing function on $\RR$ such that $s_0(t) = 1$ when $t \leq 0$ and $s_0(t) = 0$ when $t \geq 1$.
Let $s_n(t) = s_0(t-n)$.
As the Sobolev norms of the $s_n$ are independent of $n$, we write $s(t)$ for $s_n(t)$ and let context dictate $n$.

Let $\Phi \co \Delta\times\RR \to \BB^3$ be the map associated to $\exitexit^1$ constructed in Section \ref{Epstein.surfaces.section}, and let $h = \Phi^*g$ be the pullback of the hyperbolic metric via $\Phi$.
We interpolate between the hyperbolic metrics $g$ and $h$ on $\Delta \times [n, n+1]$ using the metric
\begin{equation}\label{Eta.metric.equation}
\eta = (1-s(t)) g + s(t) h,
\end{equation}
which descends to a metric on $S \times [n,n+1]$ that we also call $\eta$, see Figure \ref{GlueQF.figure}.

By \eqref{Derivative.Phi.equation}, \eqref{Eta.metric.equation}, and the Kraus--Nehari theorem, we have the following proposition, which provides Part \ref{Model.manifold.theorem.item.radius} of Theorem \ref{Model.manifold.theorem}.

\begin{proposition}\label{Bilipschitz.Jacobian.proposition}
The identity map $\iota \co (S \times [n,n+1], g) \to (S \times [n,n+1], \eta)$ is \linebreak $(1 + \BigO_s(e^{-2n}))$--bilipschitz with Jacobian determinant $1+ \BigO_s(e^{-4n})$.
\qed
\end{proposition}

To prove the other parts of Theorem \ref{Model.manifold.theorem} when $M^1 = N^1$ and $M^2 = N^2 = \qf(Y,\overline{Y})$, we begin by showing that the traceless Ricci curvature of our metric $\eta$ is on the order of $e^{-2n}$.
The region where $\eta$ is nonhyperbolic has volume on the order of $e^{2n}$, and so
it will follow that the $L^2$--norm of the traceless Ricci curvature is on the order of $e^{-n}$. 
Since we are in dimension three, having Ricci curvature on the order of $e^{-2n}$ implies that the difference between the sectional curvatures and $-1$ is on the order of $e^{-2n}$ as well.

The intuition for the estimate of the Ricci curvature is as follows.
The Ricci curvature measures the infinitesimal defect in volume of a sharp geodesic cone compared to the corresponding Euclidean cone:
the volume element of a metric $\omega$ at a point $p$ admits an asymptotic expansion in $\omega$--geodesic normal coordinates
\begin{equation}\label{General.volume.expansion}
\dee V_\omega = \left(1 - \frac{1}{6} \Ricci^\omega(\uu) \epsilon^2 + \BigO(\epsilon^3)\right) \epsilon^{2} \,\dee \epsilon \, \dee A(\uu)
\end{equation}
where $\Ricci^\omega$ is the Ricci curvature of $\omega$ considered a quadratic form, and $\dee A(\uu)$ is the canonical spherical measure on the unit tangent space $\mathrm{T}_p^1 M$, see section 3.H.4 of \cite{Gallot.Hulin.Lafontaine.2004}.
Our metric $\eta$ is obtained by gluing two hyperbolic metrics on $S \times \RR$ together fiberwise.
The original metrics on the fibers are exponentially close, as are the original normal curvatures to the fibers, and so, after interpolating to obtain $\eta$, the volumes of cones are disturbed an exponentially small amount.
One may try to make this precise using \eqref{General.volume.expansion} and Proposition \ref{Bilipschitz.Jacobian.proposition}.
This shows that the Ricci curvatures are close, but depends on the precise rate of decay of the $\BigO(\epsilon^3)$ term in $\dee V_\eta$.
Fortunately, the Ricci curvatures are fairly easy to estimate directly.

\subsection{Bounds on curvatures}\label{Curvature.section}

If $\omega = \omega_{ij} \, \dd x^i \, \dd x^j$ is a Riemannian metric, we have Christoffel symbols
\begin{equation}
	\Gamma_{ij}^\ell(\omega)	= \frac{1}{2} \omega^{k\ell}
							\left(
								\frac{\partial }{\partial x^i} \omega_{kj}
								+ \frac{\partial }{\partial x^j} \omega_{ik}
								- \frac{\partial }{\partial x^k} \omega_{ij}
							\right),	
\end{equation}
where $(\omega^{ij}) = (\omega_{ij})^{-1}$.

\subsubsection{Bounding the Ricci curvature}\label{Curvature.section.1}

The Ricci curvature tensor $\Ricci^\omega = \Ric_{ij}^\omega \, \dd x^i \, \dd x^j$ of a metric $\omega$ in coordinates $x^i$ is given by
\begin{align*}
\Ric_{ij}^\omega
& = 	\left(
		\frac{\partial \Gamma_{ij}^\ell}{\partial x^\ell} 
		- \frac{\partial \Gamma_{i\ell}^\ell}{\partial x^j} 
		+ \Gamma_{ij}^\ell \Gamma_{\ell m}^m
		- \Gamma_{i \ell}^m \Gamma_{jm}^\ell
	\right)
(\omega).
\end{align*}

\begin{theorem}\label{RicciCurvatureEstimate.theorem}
If $\uu$ is an $\eta$--unit vector at $(z,t)$ in $S \times [n,n+1]$, then 
\begin{equation}
\Ricci^\eta(\uu) - \Ricci^g(\uu) = \mathcal{O}(e^{-2n}).
\end{equation}
\end{theorem}
\noindent Theorem \ref{RicciCurvatureEstimate.theorem}  follows immediately from the following theorem.
\begin{theorem}\label{G.ETA.proximity.theorem}
We have
$
\| \, \eta - g \, \|_{\mathcal{C}^2(S \times [n,n+1], \, g)} = \BigO(e^{-2n}).
$
\end{theorem}
\begin{proof}
Since the norms $\|\Schwarz \univalent(z)\|$ and arguments $\arg(\Schwarz \univalent(z))$ are smooth functions away from the zeroes of $\Schwarz \univalent$, and the set of points $(z,t)$ such that $\Schwarz \univalent(z) = 0$ is a finite set of lines in $S \times \RR$, we restrict attention to points $(z,t)$ such that $\Schwarz \univalent(z) \neq 0$.


Let $\calQ(\overline{Y})$ be the vector space of holomorphic quadratic differentials on $\overline{Y}$.
By the Kraus--Nehari theorem, the subset of $\calQ(\overline{Y})$ consisting of Schwarzian derivatives of developing maps of Kleinian projective structures on $Y$ is compact, see \cite{Bers.1970.Boundaries}.
So there is a number $\bconstant = \bconstant(Y)$ bounding the values and first few partial derivatives of the norms $\|\Schwarz \univalent(z)\|$ and arguments $\arg(\Schwarz \univalent(z))$ on the fundamental domain $\domain{Y}$.
In other words, the functions $\|\Schwarz \univalent(z)\|$ and $\arg(\Schwarz \univalent(z))$ have Sobolev norms $\| \ \cdot \ \|_{W^{2,\scriptinfinity}(\domain{Y})}$ at most $\bconstant$.
In fact, if we fix a compact subset $\mathcal{X}$ of $\Teich(\overline{Y})$ containing $\overline{Y}$, we obtain a uniform bound $\bbconstant = \bbconstant(\mathcal{X})$ on these Sobolev norms over all of $\mathcal{X}$.
As the thick part of the moduli space $\mathcal{M}(\overline{Y})$ is compact \cite{mumford}, the action of the mapping class group $\Mod(S)$ on $\Teich(\overline{Y})$ provides a bound $\bbbconstant = \bbbconstant(\chi(S), \injrad(Y))$ on these Sobolev norms over the entire thick part of $\Teich(\overline{Y})$.

The notation $\mathcal{O}(\ \cdot \ )$ will now always mean $\mathcal{O}_{\chi(S), \, \injrad(Y)}(\ \cdot \ )$.


Consider a point $(w,t)$ in $\Delta \times [0,\infinity)$ with $\Schwarz \univalent(w) \neq 0$.
Composing $\univalent$ on both sides with M\"obius transformations, we may assume that $w = 0$ and $M_{\univalent(0)} = \mathrm{Id}$.

We work with a small ball $\ball \subset \domain{Y} \subset \Delta$ centered at $0$ containing no zeroes of $\Schwarz \univalent$.

Let $z$ be a point of $\ball$, and let $\psi_z$ be the hyperbolic element of $\PSL_2\CC$ stabilizing $\Delta$ that carries $0$ to $z$ and whose axis in $\Delta$ contains $0$.
By the invariance of the Schwarzian, we have 
\begin{equation}
\Schwarz (\univalent \circ \psi_z) (0) 
	= \Schwarz \univalent (\psi_z(0)) \psi_z'(0)^2
	= \Schwarz \univalent (z) \psi_z'(0)^2.
\end{equation}
Note that $\psi_z'(0) = 1/(1-|z|^2)$.
We may postcompose $\univalent \circ \psi_z$ with a M\"obius transformation to ensure that the osculating M\"obius transformation to $\univalent \circ \psi_z$ at $0$ is the identity, and this has no effect on the Schwarzian.
Let
\[
	\theta_z = \arg(\Schwarz \univalent(z)) - \arg\left(\psi_z'(0)^2\right).
\]
A change of variables allows us to assume that $\theta_0 = 0$.
Let
\begin{equation}
	A_z =
	\left(
\begin{tikzpicture}[>=to, line width = .075em, baseline=(current bounding box.center)]
\matrix (m) [matrix of math nodes,  text height=1.5ex, text depth=0.25ex]
{
\cos(\theta_z) & - \sin(\theta_z) & 0 \\
\sin(\theta_z) & \cos(\theta_z) & 0 \\
0 & 0 & 1 \\
};
\end{tikzpicture}
	\right).
\end{equation}

Consider coordinates $u^1 =  \cosh (t) \, x$, $u^2= \cosh (t) \, y$, and $u^3=t$.
In these coordinates, the metric $g$ is given by
\[
	g = (g_{ij}) =
		\left(
\begin{tikzpicture}[>=to, line width = .075em, baseline=(current bounding box.center)]
\matrix (m) [matrix of math nodes,  text height=1.5ex, text depth=0.25ex]
{
\poincare^2  & 0 & 0 \\
0 & \poincare^2  & 0 \\
0 & 0 & 1 \\
};
\end{tikzpicture}
		\right),
\]
where $\poincare = 2/(1-|z|^2)$.

By \eqref{Basis.equation}, we have
\begin{equation}\label{}
	\Der \Phi \big|_{(z,t)} =
	A_z
	\left(
\begin{tikzpicture}[>=to, line width = .075em, baseline=(current bounding box.center)]
\matrix (m) [matrix of math nodes,  text height=1.5ex, text depth=0.25ex]
{
1 + \frac{\scalednorm{\Schwarz \univalent(z)}}{e^t \cosh t} & 0 & 0 \\
0 & 1 - \frac{\scalednorm{\Schwarz \univalent(z)}}{e^t \cosh t} & 0 \\
0 & 0 & 1 \\
};
\end{tikzpicture}
	\right)
	A_z^{-1}
\end{equation}
in $\ball \times \RR$ with respect to the orthonormal basis $\frac{1}{\lambda}\dduone$, $\frac{1}{\lambda}\ddutwo$, $\dduthree$ to $\HH^3$ at $(z,t)$ and an orthonormal basis at $\Phi(z,t)$.
Note that $(g_{ij})$, a scalar matrix when restricted to the subspace $\dduone$ and $\ddutwo$, commutes with the matrix $\Der \Phi^T = \Der \Phi$.

Writing the metric $\eta = (1-s(t))\, g + s(t)\, h$ in these coordinates, we have
\begin{align}
	(\eta_{ij})
		& = (1-s(t))(g_{ij}) + s(t)	\, \Der \Phi^T \cdot (g_{ij}) \cdot \Der \Phi \\
		& = (1-s(t))(g_{ij}) + s(t)	\,  \cdot (g_{ij}) \cdot (\Der \Phi)^2.
\end{align}
Expanding this, we have
\begin{equation}\label{Difference.of.metrics.x.equation}
	(\eta_{ij})= 
	(g_{ij}) +
	\frac{1}{e^{t}\cosh t}
	\cdot 
	s(t) \lambda^2 \scalednorm{\Schwarz \univalent(z)}
	\cdot
	P,	
\end{equation}
where
\begin{equation}\label{Difference.of.metrics.x.P}
	P= 
	\left(
\begin{tikzpicture}[>=to, line width = .075em, baseline=(current bounding box.center)]
\matrix (m) [matrix of math nodes,  text height=2.5ex, text depth=0.25ex]
{
  2 \cos(2\theta_z) + \frac{\scalednorm{\Schwarz \univalent(z)}}{e^t \cosh t}
	& 	\ \ \ 2\sin(2\theta_z) 
		 & 0 \\
2\sin(2\theta_z) 
	& -2\cos(2\theta_z) + \frac{\scalednorm{\Schwarz \univalent(z)}}{e^t \cosh t}   & 0 \\
0 & 0 & 0 \\
};
\end{tikzpicture}
	\right)
	.
\end{equation}
\begin{proposition}\label{Metrics.close.proposition}
For all $i$, $j$, $k$, and $\ell$, we have
\[
	\left| \,
		\eta_{ij} - g_{ij}
	\, \right|
	 = \BigO(e^{-2n}),
\]
\[
	 \left| \,
		\ddxk \eta_{ij} - \ddxk g_{ij}
	\, \right|
	 = \BigO(e^{-2n}),
\]
and
\[
	 \left| \,
		\ddxl \ddxk \eta_{ij} - \ddxl \ddxk g_{ij}
	\, \right|
	 = \BigO(e^{-2n}),
\]
in coordinates $u^1 =  \cosh(t) \, x$, $u^2 =  \cosh(t) \, y$, and $u^3 = t$ on $S \times [n,n+1]$.
\end{proposition}
\begin{proof}
The estimates hold on $\ball$ by inspection of \eqref{Difference.of.metrics.x.equation} and \eqref{Difference.of.metrics.x.P}, and, since the zeroes of $\Schwarz \univalent$ are isolated and our metrics are smooth, they hold on all of $S \times [n,n+1]$ by continuity.
\end{proof}

Now, in the coordinates $u^i$, the first few derivatives of the $g_{ij}$ and $g^{ij}$ are $\BigO(1)$, and so Proposition \ref{Metrics.close.proposition} implies that $\| \, \eta - g \, \|_{\mathcal{C}^2(S \times [n,n+1], \, g)} = \BigO(e^{-2n})$, by definition of the $\mathcal{C}^2$--norm.

This completes the proof of Theorem \ref{G.ETA.proximity.theorem} (and hence of Theorem \ref{RicciCurvatureEstimate.theorem}).
\end{proof}

\subsubsection{$L^2$--norm of the traceless Ricci curvature}\label{Curvature.section.2}

To estimate the $L^2$--norm of the traceless Ricci curvature of $\eta$,
we begin by estimating the volume of the nonhyperbolic part.

\begin{lemma}\label{VolumeBound.lemma}
We have
\begin{equation}
\int_{S \times [n,n+1]} 1 \ \dee V_\eta
\leq -\Cconstant\chi(S) e^{2 n}.
\end{equation}
\end{lemma}
\begin{proof}
Let $\iota \co (S \times [n, n+1], g) \to (S \times [n, n+1], \eta)$ be the identity map.
By \eqref{Derivative.Phi.equation}, the Jacobian determinant of $\iota$ at $(z,t)$ is
\begin{equation}
|\ \mathrm{Jac}\ \iota \,| = 1 - \left(\frac{\scalednorm{\Schwarz \univalent(z)}}{e^t \cosh t}\right)^{\!\!2}.
\end{equation}

So
\begin{align*}
\int_{S \times [n,n+1]} 1 \ \dee V_\eta
		& = \int_{S \times [n,n+1]} |\ \mathrm{Jac}\ \iota \,| \ \dee V_g
		\\
		& \leq \int_{S \times [n,n+1]} 1 \ \dee V_g
		\\
		& = \int_{\domain{Y} \times [n,n+1]} \sqrt{\det g\ } \ \dee x  \, \dee y \, \dee t
		\\
		& =	
			\int_{\domain{Y} \times [n,n+1]} \ \poincare^2 \cosh^2 t  \ \dee x  \, \dee y \, \dee t
		\\
		& =
			\int_S \left(\int_{[n,n+1]}  \cosh^2 t  \ \dee t  \right)\dee A_Y
		\\
		&	\leq -2\pi\chi(S) \, e^{2n+2}
		\\
		& \leq -  18 \pi \chi(S) \, e^{2n}
		. \qedhere
\end{align*}
\end{proof}

For an $\eta$--unit vector $\uu$ at $(z,t)$, we have
\begin{equation}
\Ricci^\eta(\uu) + 2\eta(\uu) = \Ricci^g(\uu) + 2 g(\uu) + \BigO(e^{-2t}) =   \BigO(e^{-2t}) ,
\end{equation}
by Theorem \ref{RicciCurvatureEstimate.theorem}.
So there is a constant $\Bconstant = \Bconstant(\chi(S), \injrad(Y))$ such that
\begin{equation}\label{RicciCurvaturePointwiseBound.equation}
\| \Ricci^\eta + 2\eta \|_\eta \leq \Bconstant e^{-2n}.
\end{equation}

\begin{lemma}\label{RicciIntegral.lemma}
We have
\begin{equation}\label{RicciIntegral.estimate}
	\sqrt{
		\int_M \| \Ricci^\eta + 2\eta \|_\eta^2 \ \dee V_\eta \
	\ }
\leq
	-  \Cconstant \, \Bconstant \chi(S) \, e^{-n} 
	. 
\end{equation}
\end{lemma}
\begin{proof}
Since $\eta$ is hyperbolic away from $S \times [n,n+1]$, Lemma \ref{VolumeBound.lemma} and  \eqref{RicciCurvaturePointwiseBound.equation} give us
\begin{align*}
\sqrt{
\int_M \| \Ricci^\eta + 2\eta \|_\eta^2 \ \dee V_\eta \
}
& 
	\leq
	\sqrt{
		\int_{S \times [n,n+1]} \Bconstant^2 e^{-4n} \ \dee V_\eta \
	\ }
\\
&
	= 
	\Bconstant e^{-2n} 
	\sqrt{
		\int_{S \times [n,n+1]} 1 \ \dee V_\eta \
	}
\\
&
	\leq
	\Bconstant e^{-n} \sqrt{- \Cconstant \, \chi(S)\, }
\\
&
	\leq
	- \Cconstant \, \Bconstant  \chi(S) \, e^{-n} 
	. \qedhere
\end{align*}
\end{proof}

\subsubsection{Sectional curvatures}\label{Curvature.section.3}

In dimension three, the sectional curvatures are determined by the Ricci curvatures.  
More specifically, if $\uu$, $\vv$, and $\ww$ are orthonormal tangent vectors at a point in a $3$--manifold, we have
\begin{equation}\label{SectionalFromRicci.equation}
2K(\uu,\vv) = \Ricci(\uu) - \Ricci(\ww) + \Ricci(\vv).
\end{equation}
Theorem \ref{RicciCurvatureEstimate.theorem} and \eqref{SectionalFromRicci.equation} give us
\begin{equation}
2K^\eta(\uu,\vv) \big|_{(z,t)} = -2 + 2 - 2 + \BigO(e^{-2t})
\end{equation}
for any $\eta$--orthonormal vectors $\uu$ and $\vv$ at any $(z,t)$ in $\Delta \times [n,n+1]$.
So
\begin{equation}\label{Sectional.estimate}
K^\eta(\uu,\vv) \big|_{(z,t)} =  - 1 + \BigO(e^{-2n})
\end{equation}
for all $(z,t)$, since $K^\eta = -1$ on $M - (\Delta \times [n,n+1])$.
So there is a constant $\Dconstant = \Dconstant(\chi(S),\injrad(Y)) \geq \Bconstant$ such that
\begin{equation}\label{Sectional.precise.estimate}
-1 - \Dconstant e^{-2n} \leq K^\eta \leq -1 + \Dconstant e^{-2n}.
\end{equation}

Let $\Econstant = \max\{- \Cconstant \, \Bconstant \chi(S) , \, \Dconstant, 1\}$.

Setting $\AAconstant =  \Econstant$ completes the proof of Theorem \ref{Model.manifold.theorem} in the case when $M^1 = N^1$ and that $M^2 = N^2 = \qf(Y,\overline{Y})$.

\subsection{The proof of Theorem \ref{Model.manifold.theorem} in the remaining cases}\label{General.case.section}

Since the gluing takes place locally on the region $\exitexit^2_{[n,n+1]}$, the proof above provides all of the cases where $N^2 = \qf(Y,\overline{Y})$.
The cases when  $N^1 = \qf(Y,\overline{Y})$ are then obtained by replacing $s(t)$ with $1-s(t)$ in the proof.
The general case is then obtained as follows.
Let $\exitexit$ be the $\overline{Y}$--end of $\qf(Y,\overline{Y})$.
First glue $\mathcal{A}^1_{n+1}$ to $\exitexit_{[n,\smallinfinity)} \subset \qf(Y,\overline{Y})$ along $\exitexit_{[n,n+1]}$ as above.
Then glue $\qf(Y,\overline{Y}) - \exitexit_{[n+3,\smallinfinity)}$ to $\mathcal{B}^2_{n+2}$ along $\exitexit_{[n+2,n+3]}$ as above.
The resulting manifolds both contain an isometric copy of $\exitexit_{[n+1,n+2]}$
, and a simple cut and paste completes the proof of the general case.
\end{proof}

\section{Curvatures of surfaces and normal projections}

Letting $w = x_1 + i x_2$, we compactify hyperbolic space by attaching the Riemann sphere $\widehat{\CC}$ via the upper half-space model
\[
	\HH^3 = 
			\{ 
				(x_1,x_2,x_3) \in \RR^3 \ | \ x_3 > 0
			\}.
\]

Let $\mathcal{F}$ be a smooth surface in $\HH^3$ equipped with a smooth unit normal field, and let $q$ be in $\mathcal{F}$.
Applying an element of $\PSL_2(\CC)$, we assume that $q=(0,0,1)$, that the unit normal to $\mathcal{F}$ at $q$ is $-\mathbf{k}$, and that the principal directions at $q$ are $\ii$ and $\jj$. 
Let $\nu$ be the normal projection (or \textit{Gauss map}) of $\mathcal{F}$ to $\widehat{\CC}$ that sends each point of $\mathcal{F}$ to the point of $\widehat{\CC}$ at the end of the geodesic ray given by our normal field. 
Picking an orthonormal basis for $\mathrm{T}_q\mathcal{F}$ along its principal directions and the usual basis for $\mathrm{T}_0 \widehat{\CC}$, the derivative of $\nu$ at $q$ is given by 
\begin{equation}\label{NormalProjectionDerivative}
	\Der \nu_q 
		= \begin{pmatrix} \frac{1 + \kappa_1}{2}  	&   0 \\
  	0 									&  \frac{1 + \kappa_2}{2}
\end{pmatrix}
\end{equation}
where the $\kappa_i$ are the principal curvatures of $\mathcal{F}$ at $q$.
Our convention is that normal curvatures are \textit{positive} when the surface is curving \textit{away} from the normal vector.

\begin{lemma}\label{ConvexSurfaceQuasiconformal.lemma}
Let $S \times [0,\infty)$ be a closed smoothly concave neighborhood of a convex cocompact end $E$ of a hyperbolic manifold and let $Z$ be the conformal boundary at $E$.
If the principal curvatures of $\mathcal{G} = S \times \{0\}$ are within $\epsilon$ of $1$ for some $0 < \epsilon < 1$, then $\mathcal{G}$ and $Z$ are $(1+\epsilon)^2$--quasiconformal.
\end{lemma}
\begin{proof}
Lift $\mathcal{G}$ to a surface $\widetilde{\mathcal{G}}$ in $\HH^3$ and normalize as above so that the derivative of the normal projection at a point $q$ in $\widetilde{\mathcal{G}}$ is
\[
\Der \nu_q = 	
			\begin{pmatrix} \frac{1 + \kappa_1(q)}{2}  &   0 \\
						  0 &  \frac{1 + \kappa_2(q)}{2}
  			\end{pmatrix}
\]
where the $\kappa_i(q)$ are the principal curvatures of $\widetilde{\mathcal{G}}$ at $q$.
The usual Euclidean metrics on the tangent spaces $\mathrm{T}_q \widetilde{\mathcal{G}}$ and $\mathrm{T}_0 \widehat{\CC}$ are conformally compatible with the Riemannian metrics on $ \widetilde{\mathcal{G}}$ and $\widetilde{Z}$, respectively.
Since the dilatation of the linear map $\Der \nu_q$ is at most $(1 + \epsilon/2)/(1-\epsilon/2) < (1+\epsilon)^2$ at the origin, the dilatation of the quasiconformal map $\nu$ is no more than $(1+\epsilon)^2$ at $q$, see Chapter 1 of \cite{Ahlfors.1966}.
\end{proof}

\begin{lemma}\label{PrincipalCurvaturesCloseToOne.lemma} 
Let $\qf(W,Z)$ be a quasifuchsian manifold, let $\Schwarz \univalent$ be the Schwarzian derivative of the developing map $\univalent \co \Delta \to \UU^Z$,
and let $\EP{}{t}$ be the Epstein surface at time $t$ in the $Z$--end of $\qf(W,Z)$.
If $t \geq \log 9$, then the principal curvatures $\kappa_\pm(z,t)$ of $\EP{}{t}$ at $(z,t)$ satisfy
\begin{equation}\label{PrincipalCurvaturesCloseToOne.equation}
	\left| 
		\kappa_\pm(z,t) - 1
	\right|
\leq 9 e^{-2t}. 
\end{equation}
\end{lemma}
\begin{proof}
This follows immediately from \eqref{PrincipalCurvatures} and the Kraus--Nehari Theorem.
\end{proof}

\section{Proof of Theorem \ref{DiameterBound.Generalization.theorem}}\label{Proof.section}

\begin{proof}
Let $M^W \cong M$ be a convex cocompact hyperbolic manifold with conformal incompressible boundary $W$, let
$
p \co \qf(W, \overline{Y}) \to M^W
$
be the covering map corresponding to $W$, and assume that $p$ embeds the $d$--neighborhood of the convex core of $\qf(W,\overline{Y})$  isometrically into $M^W$.
Let $n = \floor{\dconstant} - 8$.
We make the crude choice of $8$ to be sure that $S \times [n,n+3]$ avoids the thin part of $M^W$.

First assume that $X$ is such that the $2n$--neighborhood of the convex core of $\qf(X,\overline{Y})$ together with the $\overline{Y}$--end of $\qf(X,\overline{Y})$ isometrically embeds into a convex cocompact hyperbolic $3$--manifold $N^1$ with conformal boundary $\overline{Y}$, see Figure \ref{GlueRXMY.figure}. 
Letting $M^1 = N^1$, $N^2 = \qf(W,\overline{Y})$, and $M^2 = M$,  let $\Manifold$, $\eta$, and $\AAconstant = \AAconstant(\chi(S), \injrad(Y))$ be as in Theorem \ref{Model.manifold.theorem}.
Note that while the topology of $\Manifold$ depends on $X$, our estimates do not.
Since $\Manifold$ is closed, it follows from Tian's theorem that there is a $\Tconstant = \Tconstant(\chi(S), \injrad(Y)) \geq \log 9$ such that, when $n \geq \Tconstant$, the metric $\eta$ on $\Manifold$ is within $\TianC\AAconstant e^{-n} < 1/2$ of a hyperbolic metric $\rho$ on $\Manifold$ in the $\mathcal{C}^2$--norm.\footnote{Note that $\Tconstant$ does depend on $\injrad(Y)$, for we must be at a certain depth in the collar to ensure that the traceless Ricci curvature of our metric vanishes on the thin parts.}
In particular, an $\eta$--unit vector has $\rho$--length within $\TianC\AAconstant e^{-n}$ of $1$. 
The metrics are then $(1 + \constantA e^{-n})$--bilipschitz for $\constantA=\TianC\AAconstant$.
We summarize this discussion in a proposition.
\begin{proposition}\label{Promixity.bilipschitz.proposition}
The metrics $\eta$ and $\rho$ satisfy
\begin{equation}\label{Proximity.equation}
	\| \eta - \rho \|_{\mathcal{C}^2(\Manifold,\eta)} \leq \constantA e^{-n}
\end{equation}
and are therefore $(1 + \constantA e^{-n})$--bilipschitz.
\qed
\end{proposition}

\begin{figure}
\bigskip
\bigskip
\bigskip
\bigskip
\bigskip
\medskip
\quad \quad
\quad \quad
\resizebox{.7\textwidth}{!}{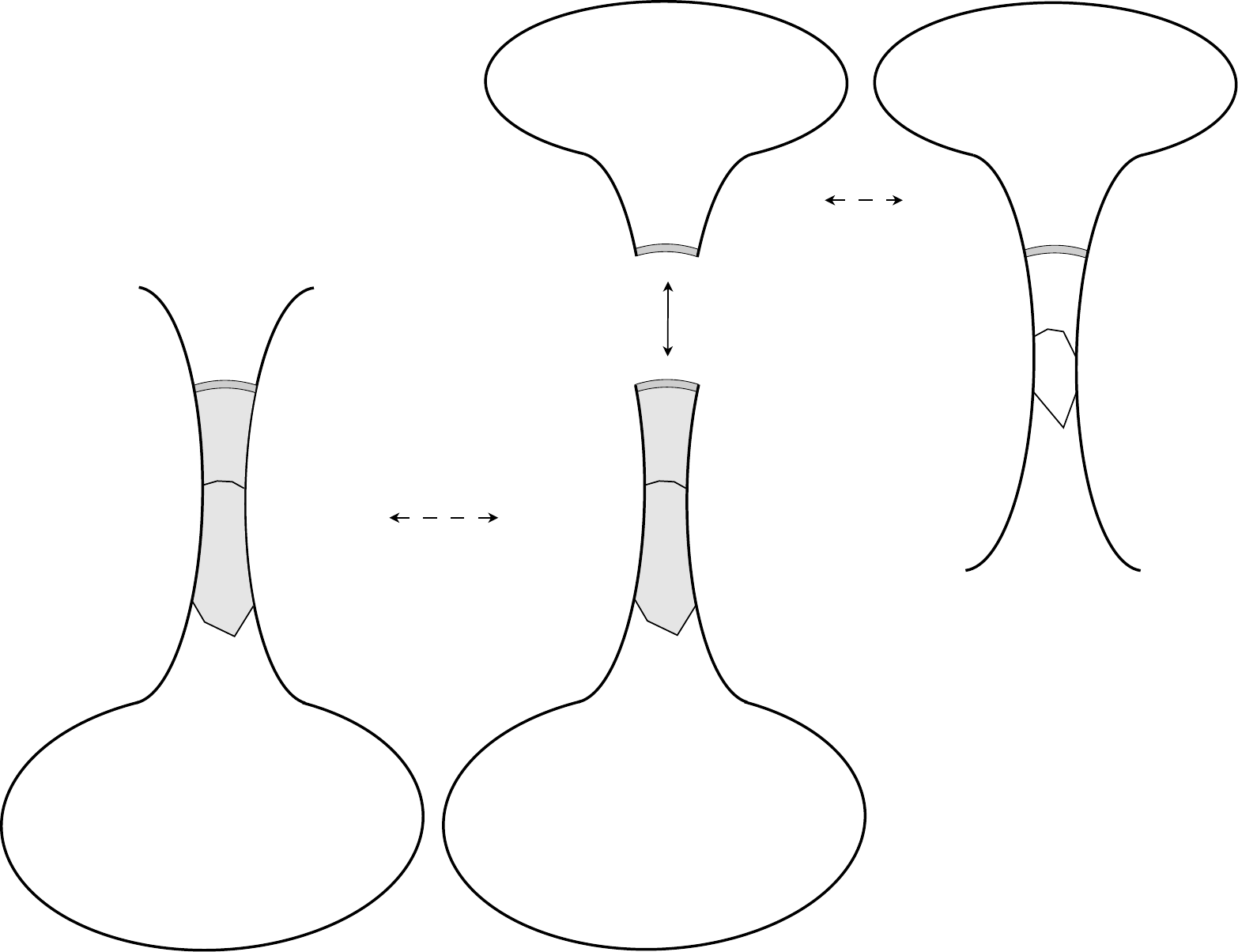}
\caption{Building $\Manifold$ with the metric $\eta$.}\label{GlueRXMY.figure}
\end{figure}

Let $\Manifold_\eta = (\Manifold, \eta)$ and $\Manifold_\rho = (\Manifold, \rho)$, and henceforth identify their tangent spaces.

The cover of $\Manifold_\rho$ corresponding to $\partial M$ is a quasifuchsian manifold $\qf(Z,\sigma_M(Z))$, see Figure \ref{Comparison.figure}.

\begin{figure}
\bigskip
\bigskip
\bigskip
\bigskip
\smallskip
\resizebox{.85\textwidth}{!}{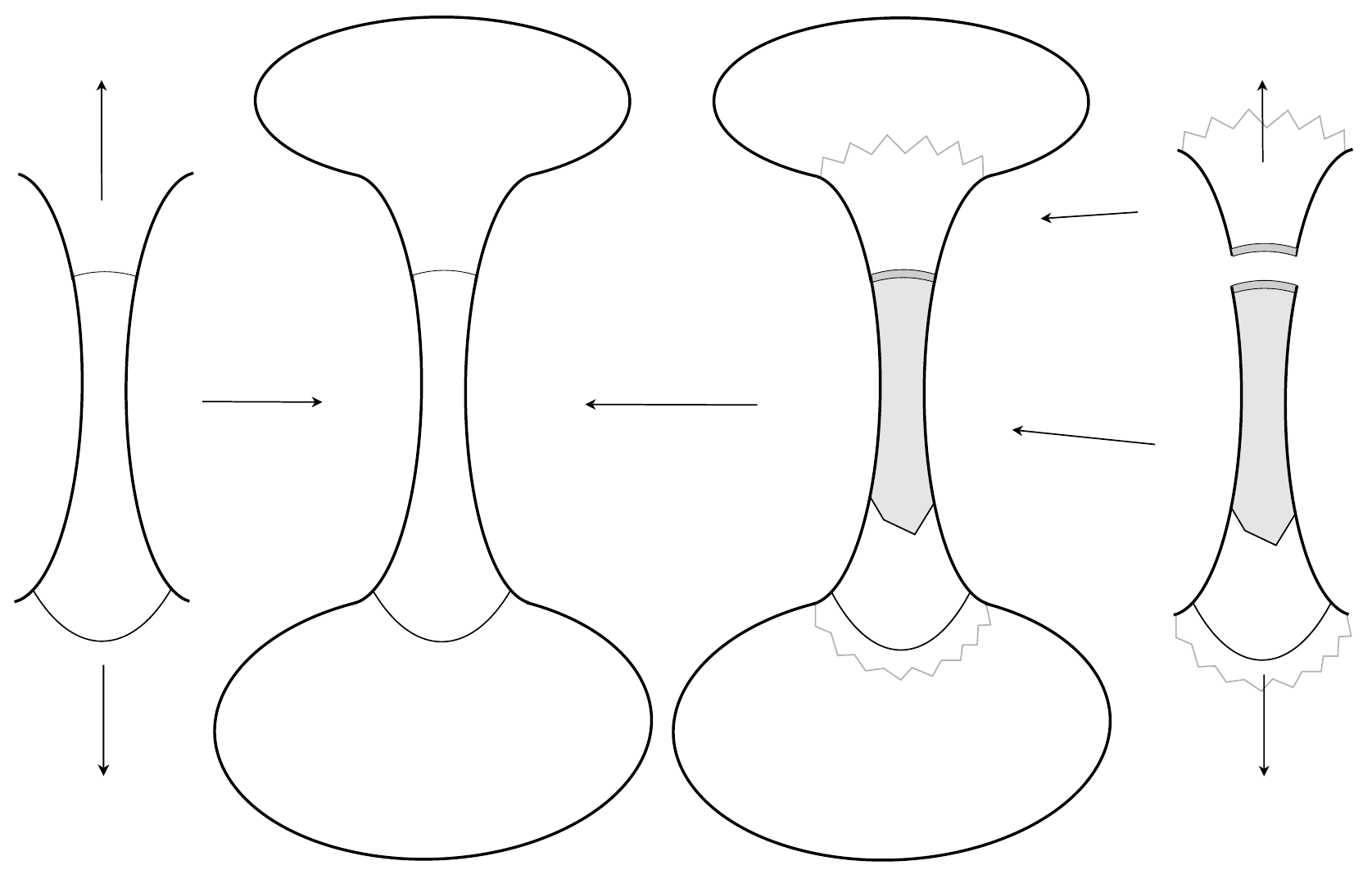}
\caption{The manifold $\Manifold$ with its two metrics $\rho$ and $\eta$. 
At left is the covering map $\qf(Z,\sigma_M(Z)) \to \Manifold_\rho$, together with the projections $\mathcal{G}_{-n} \to Z$ and $\mathcal{F}_{n} \to \sigma_M(Z)$.  
At right are partial covering maps from subsets of $\qf(Y,\overline{Y})$ and $\qf(X,\overline{Y})$ to $\Manifold_\eta$ with corresponding projections $\EP{2}{n+3} \to \overline{Y}$ and $\EP{X}{-n} \to X$.
}\label{Comparison.figure}
\end{figure}

We now show that there is a constant $\constantB = \constantB(\chi(S), \injrad(Y)) \geq \constantA$ such that the  Teichm\"uller distance between $\sigma_M(Z)$ and $\overline{Y}$ is less than $2 \constantB e^{-n}$.

Consider the surface $\EP{2}{n+3}$ in $\Manifold_\eta$.
Let $\mathcal{F}_n$ be the image of $\EP{2}{n+3}$ in $\Manifold_\rho$, see Figure \ref{Comparison.figure}.
A small neighborhood of $\mathcal{F}_n$ lifts isometrically to $\qf(Z,\sigma_M(Z))$, and we continue to use $\mathcal{F}_n$ to denote this lift.
Let $\nn$ be the field of $\eta$--unit normal vectors to $\EP{2}{n+3}$ in $\Manifold_\eta$ pointing toward $\overline{Y}$.
Let $\uu$ be an $\eta$--unit vector field tangent to $\EP{2}{n+3}$ on some open patch of $\EP{2}{n+3}$.
The normal curvature of $\EP{2}{n+3}$ along $\uu$ is given by
\begin{equation}\label{Epstein.Normal.Curvature.equation}
	\normalcurv_{\EP{2}{n+3}}\!(\uu)
		= \frac{\mathrm{I\!I}_{\EP{2}{n+3}}\!(\uu,\uu)}{\| \uu \|^2_\eta}
		= - 	\frac{\eta(\nn, \grad_\uu^\eta \uu)}{\| \uu \|^2_\eta}
		= -	\eta(\nn, \grad_\uu^\eta \uu)
\end{equation}
where $\grad^\eta$ is the Levi--Civita connection for $\eta$. 
By Lemma \ref{PrincipalCurvaturesCloseToOne.lemma}, this curvature is within $9e^{-2n}$ of $1$.

Letting $\mm$ be the $\rho$--unit normal field to $\mathcal{F}_n$ pointing toward the skinning surface $\sigma_M(Z)$, the normal curvature of $\mathcal{F}_n$ is given by
\begin{equation}\label{ImageOfEpstein.Normal.Curvature.equation}
	\normalcurv_{\mathcal{F}_n}\!(\uu)
		=\frac{\mathrm{I\!I}_{\mathcal{F}_n}(\uu,\uu)}{\| \uu \|^2_\rho}
		= -	\frac{\rho\big(\mm, \grad_\uu^\rho \uu\big)}{\| \uu \|^2_\rho}
\end{equation}
where $\grad^\rho$ is the Levi--Civita connection for $\rho$.
The proximity \eqref{Proximity.equation} of the metrics provides the following estimate.

\begin{claim}
The normal curvatures satisfy
\begin{equation}
	\Big| \,
		\normalcurv_{\EP{2}{n+3}}\!(\uu) - \normalcurv_{\mathcal{F}_n}\!(\uu) 
	\, \Big|
	= \BigO(e^{-n}).
\end{equation}
\end{claim}
\noindent\textit{Proof of claim.}
We have
$
\big\| 
		\grad_\uu^\eta \uu 
	\big\|_\eta 
	= \BigO(1),
$
and so \eqref{Proximity.equation} gives us
\begin{equation}\label{Zeta.Derivative.OmegaZetaNorm.equation}
	\big\| 
		\grad_\uu^\rho \uu 
	\big\|_\eta 
	+
	\big\| 
		\grad_\uu^\rho \uu 
	\big\|_\rho 
	= \BigO(1).
\end{equation}
Together with \eqref{Proximity.equation}, this yields
\begin{equation}
	\Big| \,
		\eta(\nn, \grad_\uu^\eta \uu)
			- \rho\big(\nn, \grad_\uu^\rho \uu\big)
	\, \Big|
= \BigO(e^{-n}).
\end{equation}

Proposition \ref{Promixity.bilipschitz.proposition}  gives us
\begin{equation}\label{mn.difference.equation}
\| \nn - \mm \|_\eta + \| \nn - \mm \|_\rho = \BigO(e^{-n}),
\end{equation}
and then \eqref{Zeta.Derivative.OmegaZetaNorm.equation} and the Cauchy--Schwarz inequality give us
\begin{align}
	\Big| \, 
		\rho\big(\nn, \grad_\uu^\rho \uu\big) 
			- \rho\big(\mm, \grad_\uu^\rho \uu\big)
	\, \Big| 
& 	\leq
		\| \nn - \mm \|_\rho    \big\| \grad_\uu^\rho \uu \big \|_\rho 
			= \BigO(e^{-n}).
\end{align} 

Now, 
\[
\Big|\| \, \uu \|^2_\rho - 1 \, \Big| \leq \constantA e^{-n} \leq \frac{1}{2},
\]
by \eqref{Proximity.equation}, and we conclude that
\begin{equation}
	\Big| \,
		\normalcurv_{\EP{2}{n+3}}\!(\uu) - \normalcurv_{\mathcal{F}_n}\!(\uu) 
	\, \Big|
	=
	\left| \,
	\frac{\eta(\nn, \grad_\uu^\eta \uu)}{\| \uu \|^2_\eta}
		- \frac{\rho\big(\mm, \grad_\uu^\rho \uu\big)}{\| \uu \|^2_\rho}
	\, \right|
	= \BigO(e^{-n})
\end{equation}
by the triangle inequality.
\quad \quad \quad \quad \quad \quad \quad \quad
\quad \quad \quad \quad \quad \quad \quad \quad
\quad \quad \quad \quad \quad \ \ \
\raisebox{2pt}{\framebox[21pt]{\scriptsize Claim}}

\medskip

So there is an $\constantB = \constantB(\chi(S), \injrad(Y)) \geq \constantA$ such that the normal curvatures of $\mathcal{F}_n$ are within $\constantB e^{-n}$ of $1$.
It follows that there is an $n_0 = n_0(\chi(S), \injrad(Y))$ such that, for $n \geq n_0$, the normal projection $\nu_n \co \mathcal{F}_n \to \sigma_M(Z)$ is defined and nonsingular. 
Lemma \ref{ConvexSurfaceQuasiconformal.lemma} tells us that $\nu_n$ is $(1 + \constantB e^{-n})^2$--quasiconformal and that $\EP{2}{n+3}$ is ${(1 + 9 e^{-2n})^2}$--quasiconformally equivalent to $\overline{Y}$.
Since the map $\EP{2}{n+3} \to \mathcal{F}_n$ is  $(1+ \constantA e^{-n})^2$--quasi-conformal, we conclude that the Teichm\"uller distance between $\sigma_M(Z)$ and $\overline{Y}$ is no more than $\log(1+ \constantA e^{-n}) + \log(1 + 9 e^{-2n}) +\log(1+ \constantB e^{-n}) \leq 3\, \constantB e^{-n}$.

We now claim that there is an $\constantC = \constantC(\chi(S), \injrad(Y)) \geq \constantB$ such that $Z$ is within $3\constantC e^{-n}$ of $X$.
To see this, let $\EP{X}{-n}$ be the Epstein surface in the $X$--end of $\qf(X,\overline{Y})$ at distance $n$.
By our assumption on $X$, a small neighborhood of this surface embeds isometrically in $N^1$, and hence $\Manifold_\eta$.
We let $\mathcal{G}_{-n}$ denote this surface considered in $\Manifold_\rho$.
By \eqref{PrincipalCurvaturesCloseToOne.equation}, the principal curvatures of $\EP{X}{-n}$ are within $9e^{-2n}$ of $1$.
So the principal curvatures of $\mathcal{G}_{-n}$ are $1 + \BigO(e^{-n})$, as in the above argument.
By Lemma \ref{ConvexSurfaceQuasiconformal.lemma}, the projection $\nu_{-n} \co \mathcal{G}_{-n} \to Z$ is $(1 + \constantC e^{-n})^2$--quasiconformal for some $\constantC$.
Since proximity of the metrics tells us that $\mathcal{G}_{-n}$ and $\EP{X}{-n}$ are $(1 + \constantA e^{-n})^2$--quasiconformal, and $\EP{X}{-n}$ is ${(1+ e^{-n})^2}$--quasiconformally equivalent to $X$, we conclude that the Teichm\"uller distance between $Z$ and $X$ is less than $3 \, \constantC e^{-n}$.

By Royden's theorem that the Teichm\"uller and Kobayashi metrics agree \cite{Royden},  the skinning map is $1$--lipschitz, as it is holomorphic.\footnote{A more direct argument may be found in Section 16.3 of \cite{Kapovich.2001}.
The argument there establishes the stronger statement that skinning maps are strictly distance--decreasing unless the $3$--manifold is an interval bundle over a surface.
See \cite{skinmcmullen} for another proof of this, as well as a proof that skinning maps of acylindrical manifolds are uniformly contracting.
}
We conclude that the distance between $\sigma_M(X)$ and $\sigma_M(Z)$ is at most $3 \, \constantC e^{-n}$, and so the distance between $\sigma_M(X)$ and $\overline{Y}$ is at most $6\, \constantC e^{-n}$.

Using circle packings with very small circles in the proof of Brooks's theorem \cite{Brooks.1986} (as done in Theorems 31 and 33 of \cite{Kent.2010}) demonstrates that any Riemann surface $V$ is within $\constantC e^{-n}$ of an $X$ such that the $2n$--neighborhood of the convex core of $Q(X,\overline{Y})$ together with the $\overline{Y}$--end of $Q(X,\overline{Y})$ embeds into a convex cocompact hyperbolic $3$--manifold $N^1$ with conformal boundary $\overline{Y}$.
Since skinning maps are $1$--lipschitz, we conclude that the diameter of $\sigma_M$ is no more than $7\, \constantC e^{-n} = 7\, \constantC e^{-(\floor{\dconstant} -8)} 
\leq 56722\, \constantC e^{-\dconstant} 
$.

Since the metrics $\eta$ and $\rho$ are within $\TianC\AAconstant e^{-n}$ of each other, the images $\mathcal{F}_t$ of the $\EP{2}{t}$ in $\Manifold_\rho$ are convex for $t$ greater than our chosen $T$.
It follows that there is a $K_0 = K_0(\chi(S),\injrad(Y))$ such that the $K_0 d$--neighborhood of the convex core of $\qf(Z, \sigma_M(Z))$ embeds in $M^Z$.
Since $X$ is within $3 \, \constantC e^{-n}$ of $Z$, Corollary B.23 and Proposition 2.16 of \cite{renorm}\footnote{Corollary B.23 says that $L$--quasiconformal conjugacies of Kleinian groups provide $L^{3/2}$--bilipschitz maps between hyperbolic manifolds, and Proposition 2.16 bounds the distance between convex cores of bilipschitz manifolds.} provide a $K_1 = K_1(\chi(S),\injrad(Y))$ such that the $K_1 d$--neighbor-hood of the convex core of $\qf(X,\sigma_M(X))$ embeds in $M^X$.
As any Riemann surface is within $\constantC e^{-n}$ of such an $X$, another application of Corollary B.23 and Proposition 2.16 of \cite{renorm} provide a $K = K(\chi(S),\injrad(Y))$ such that the $K d$--neighborhood of the convex core of $\qf(V,\sigma_M(V))$ embeds in $M^V$ for any $V$.

This completes the proof of Theorem \ref{DiameterBound.Generalization.theorem}.
\end{proof}


\bibliographystyle{plain}
\bibliography{ThickSkinned.Biblio}

\medskip

{
\footnotesize
\noindent 
\textsc{Department of Mathematics, University of Wisconsin -- Madison, Madison, WI 53706} 
\\ \noindent  \texttt{rkent@math.wisc.edu}   

\medskip
\noindent 
\textsc{Department of Mathematics, Yale University, New Haven, CT 06520} 
\\ \noindent  \texttt{yair.minsky@yale.edu}   

}

\end{document}